\documentclass{amsart}
\usepackage{amsfonts}
\usepackage{amsmath,amsthm}
\usepackage{amsfonts,amssymb}
\usepackage{CJK}
\usepackage{enumerate}
\newtheorem{thm}{Theorem}[section]
\newtheorem*{thm*}{Theorem}

\newtheorem*{nota}{Notation}

\newtheorem{lem}[thm]{Lemma}
\newtheorem{prop}[thm]{Proposition}

\newtheorem{defn}[thm]{Definition}

\theoremstyle{remark}
\newtheorem{rem}[thm]{Remark}
\newtheorem{exam}[thm]{Example}
\numberwithin{equation}{section}

\newcommand{\al}{\alpha}

\def\vz{\varepsilon}

\def\({\Bigl(}
\def \){ \Bigr)}

\def\x{{\bf x}}
\def\y{{\bf y}}

 \def\RR{{\mathbb R}}

\def\ss{{\Bbb S}^{d}}

\def\x{{\bf x}}
\def\y{{\bf y}}

\def\y{{\bf y}}

\begin{document}
\def\RR{\mathbb{R}}
\def\Exp{\text{Exp}}
\def\FF{\mathcal{F}_\al}

\title[Weighted least $\ell_p$ approximation]{Weighted least $\ell_p$ approximation  on compact Riemannian manifolds}

\author[J. Li]{Jiansong Li} \address{ School of Mathematical Sciences, Capital Normal
University, Beijing 100048,
 China}
\email{2210501007@cnu.edu.cn}

\author[Y. Ling]{Yun Ling} \address{ School of Mathematical Sciences, Capital Normal
University, Beijing 100048,
 China.}
 \email{18158616224@163.com}

 \author[J. Geng]{Jiaxin Geng} \address{ School of Mathematical Sciences, Capital Normal
University, Beijing 100048,
 China.}
 \email{gengjiaxin1208@163.com}

\author[H. Wang]{Heping Wang}
 \address{ School of Mathematical Sciences, Capital Normal
University, Beijing 100048,
 China}
\email{wanghp@cnu.edu.cn}

\date{\today}


\begin{abstract}Given a sequence of Marcinkiewicz-Zygmund inequalities in $L_2$ on a compact
space,
 Gr\"ochenig  in \cite{G} discussed  weighted least squares approximation
 and least squares quadrature. Inspired by this work, for all $1\le p\le\infty$, we develop  weighted
 least $\ell_p$ approximation induced by a sequence of Marcinkiewicz-Zygmund inequalities
 in $L_p$ on a compact  smooth Riemannian manifold $\Bbb M$  with  normalized Riemannian measure (typical examples are the torus
and the sphere). In this paper we derive corresponding
approximation theorems with the error measured in $L_q,\,1\le
q\le\infty$,
 and least quadrature errors for both Sobolev spaces $H_p^r(\Bbb M), \, r>d/p$ generated by eigenfunctions associated with the Laplace-Beltrami
 operator and Besov spaces $B_{p,\tau}^r(\Bbb M),\, 0<\tau\le \infty, r>d/p $ defined by best polynomial approximation.
Finally, we  discuss the optimality of the obtained results by
giving  sharp estimates of sampling numbers and optimal quadrature
errors for the aforementioned spaces.
\end{abstract}

\maketitle
\input amssym.def

\section{Introduction}\label{sec1}

This paper is concerned with constructive polynomial approximation
on a  compact  Riemannian manifold  that uses function values (the
samples) at selected well-distributed points (standard
information). Here the {\bf well-distributed} points indicate that
those points constitute an $L_p$-Marcinkiewicz-Zygmund  family.

For $1\le
p\le\infty$, an $L_p$-Marcinkiewicz-Zygmund  family on a compact space is a double-indexed set of points $\{{\bf x}_{n,k}\}$
with a family of positive numbers $\{\tau_{n,k}\}$ such that the
sampled $\ell_p$ norm of the $n$-th layer
$\left(\sum_{k}\tau_{n,k}|Q({\bf x}_{n,k})|^p\right)^{1/p}$ is
equivalent to the $L_p$ norm for  all ``polynomials" $Q$ of degree at most
$n$  with uniform constants. $L_p$-Marcinkiewicz-Zygmund
families on various compact spaces have been considered in large
number of papers and widely used in approximation theory, number
theory, signal processing, etc. The existence of
$L_p$-Marcinkiewicz-Zygmund families has been established in
many compact spaces, such as the torus, the unit sphere, unit
ball, and more general compact manifolds. Additionally, necessary
and sufficient density conditions for $L_p$-Marcinkiewicz-Zygmund
families on these domains have been discussed. See e.g.
\cite{AMO,BO,DaX,FM,M,MO,MP,OP,OS,PX,W1}.

Based on $L_2$-Marcinkiewicz-Zygmund families on a compact space
$M$, Gr\"ochenig in \cite{G} explored the weighted least squares
approximation operators and least squares quadrature rules.
Remarkably, he derived approximation theorems and least squares
quadrature errors for Sobolev spaces. Specifically, the following
results were obtained.

\begin{thm*}[\cite{G}]Assume that ${M}$ is a compact space with a probability measure $\nu$. Let $L_{n}$ be
the weighted least squares approximation operator and $I_n$ be the
least squares quadrature rule induced by an
$L_2$-Marcinkiewicz-Zygmund family $\mathcal{X}$ on $M$ with
associated weights $\tau$ and global condition number $\kappa$.
Under the hypothesis of Weyl's law, consider the Sobolev space
$H^r({M})$ generated by eigenfunctions associated to an unbounded
positive operator $A$ on $L_2({M})$. We have for all $f\in H^r(M)$
with $r>d/2$,
\begin{equation}\label{1.1}
  \|f-L_{n}(f)\|_{L_2(M)}\le
c(1+\kappa^2)^{1/2}n^{-r+d/2}\|f\|_{H^r({M})},
\end{equation}
and
$$
  \left|\int_{{M}}f({\bf x}){\rm d}\nu({\bf x})-I_n(f)\right|\le
c(1+\kappa^{1/2})n^{-r+d/2}\|f\|_{H^r({M})},
$$
where $c>0$ is independent of $n$, $\kappa$, $f$, or $\mathcal{X}$. See \cite{G} for details.
\end{thm*}

Several notable points in \cite{G} should be mentioned.
\begin{itemize}
  \item Firstly, the least squares quadrature rules were derived from Marcinkiewicz-Zygmund inequalities in $L_2$ by means of frame theory,
   whereas traditional quadrature rules were often associated with Marcinkiewicz-Zygmund inequalities in $L_1$.
  \item Secondly, the obtained error for the Sobolev spaces with smoothness index $r$ is $\mathcal{O}(n^{-r+d/2})$,
  which is almost optimal, and the constants depend only on the global condition number $\kappa$ of the $L_2$-Marcinkiewicz-Zygmund family.
  \item Thirdly, the proofs were based on the hypothesis of Weyl's law, which is related to the critical Sobolev
  exponent and leads to explicit and transparent error estimates.
\end{itemize}

However, the generality of the hypothesis results in suboptimal
error estimates. Additionally, Gr\"ochenig remarks in \cite{G}
that this technique  is certainly limited, which works only for
$L_2$ since it utilizes orthogonal projections, and may require
different techniques for general $p$-norms.

On the unit sphere $\Bbb S^d$ in $\Bbb R^{d+1}$, a special case of
compact Riemannian manifolds, Lu and Wang \cite{LW} used the
orthogonality of projections and the regularity of discrete
measures  to achieve optimal errors of the weighted least squares
approximation  and least squares quadrature, which can be stated
as follows. However the technique in \cite{LW}  works only for
$L_2$.

\begin{thm*}[\cite{LW}] Let $\Bbb S^{d}$ be the unit sphere in $\Bbb R^{d+1}$ with normalized surface measure
$\sigma$. Let $L_{n}$ be the weighted least squares approximation
operator and $I_n$ be least squares quadrature rule induced by an
$L_2$-Marcinkiewicz-Zygmund family $\mathcal{X}$ on $\ss$ with
associated weights $\tau$ and global condition number $\kappa$.
Consider the Sobolev space $H^r(\Bbb S^{d})$ generated by
eigenfunctions associated to the Laplace-Beltrami operator on
$\Bbb S^d$. We have for all $f\in H^r(\Bbb S^d)$ with $r>d/2$,
$$ \|f-L_{n}(f)\|_{L_2(\Bbb S^d)}\le
c(1+\kappa^2)^{1/2}n^{-r}\|f\|_{H^r(\Bbb S^{d})},$$ and
$$\left|\int_{\Bbb S^{d}}f({\bf x}){\rm d}\sigma ({\bf x})-I_n(f)\right|\le
c(1+\kappa^2)^{1/2}n^{-r}\|f\|_{H^r(\Bbb S^{d})},$$ where $c>0$
is independent of $n$, $\kappa$, $f$, or $\mathcal{X}$. See \cite{LW} for details.
 \end{thm*}

The purpose of this paper is to extend these results to  general
$p$-norm on compact Riemannian manifolds $\Bbb M$. Our
contributions can be summarized in five aspects.
\begin{itemize}
  \item Firstly, instead of relying on the orthogonality of projections, we utilize the minimality of the weighted least $\ell_p$
  approximation operators along with the regularity of discrete measures (i.e. Lemma \ref{lem2.1}) to derive approximation theorems and
  least squares quadrature errors.
  \item Secondly, our error estimates, $\mathcal{O}(n^{-r+d(1/p-1/q)_+})$ for the weighted least $\ell_p$ approximation operators measured in
  $L_q$ and $\mathcal{O}(n^{-r})$ for least squares quadrature rules, are asymptotically
  optimal. Here, $a_+:=\max\{a,0\}$ for $a\in\Bbb R$.
  This holds particularly true for scenarios involving the $d$-torus $\Bbb T^d$ and the  sphere $\Bbb S^d$.
  \item Thirdly, we improve the dependency on the global condition number $\kappa$ of the constant in \eqref{1.1} by replacing $(1+\kappa^2)^{1/2}$
  with $1+\kappa^{1/2}$, resulting in a slight enhancement when $\kappa$ is large.
  \item Fourthly, we also obtain  sharp estimates of (linear) sampling numbers and optimal quadratures to elaborate the optimality of obtained errors.
  \item Finally, we obtain the  above results not only for Sobolev spaces $H_p^r(\Bbb M),\,r>d/p$ generated by eigenfunctions associated with the Laplace-Beltrami
 operator and Besov spaces $B_{p,\tau}^r(\Bbb M),\, 0<\tau\le \infty,\, r>d/p $ defined by best ''polynomial" approximation.
\end{itemize}

 In what follows, we always assume that $\Bbb M$ is a compact connected smooth $d$-dimensional Riemannian manifold without boundary
 equipped with  normalized Riemannian measure $\mu$, and $X_p^r(\Bbb M)$ denotes Sobolev spaces $H_p^r(\Bbb M)$ or Besov spaces
 $B_{p,\tau}^r(\Bbb M)$. See Sect.\ref{sect2} for definitions. The main results can be formulated as below.

\begin{thm}\label{thm1.1}Let $1\le p,q\le \infty$, $0<\tau\le\infty$, and $r>d/p$. Suppose that $\mathcal{X}$
is an $L_{p}$-Marcinkiewicz-Zygmund family  on $\Bbb M$ with associated weights $\tau$ and global condition number
$\kappa$, and $L_{n,p}^{\Bbb M}$ is the weighted least $\ell_p$
approximation operator induced by $\mathcal{X}$. If
 $f\in X_{p}^r(\Bbb M)$, then for   $1\le p<\infty$,
  $$
  \|f-L_{n,p}^{\Bbb M}(f)\|_{L_q(\Bbb M)}\le C(1+\kappa^{1/p})n^{-r+d\left(\frac1p-\frac1q\right)_+}\|f\|_{X_{p}^r(\Bbb M)},
  $$
and  for $p=\infty$,
$$
  \|f-L_{n,\infty}^{\Bbb M}(f)\|_{L_q(\Bbb M)}\le C(1+\kappa)n^{-r}\|f\|_{X_{\infty}^r(\Bbb M)},
$$
where $C>0$ is independent of $n$, $\kappa$, $f$, or
$\mathcal{X}$.
\end{thm}

 We can also use the weighted least $\ell_q$
approximation operator to obtain approximation theorem as follow.

 \begin{thm}\label{rem1.1}Let $1\le p,q\le \infty$, $0<\tau\le\infty$, and $r>d\max\{1/p,1/q\}$. Suppose that $\mathcal{X}$
is an $L_{q}$-Marcinkiewicz-Zygmund family on $\Bbb M$ with
associated weights $\tau$ and global condition number $\kappa$,
and $L_{n,q}^{\Bbb M}$ is the weighted least $\ell_q$
approximation operator induced by $\mathcal{X}$. If
 $f\in X_{p}^r(\Bbb M)$, then for $1\le q<\infty$,
  $$
  \|f-L_{n,q}^{\Bbb M}(f)\|_{L_q(\Bbb M)}\le C(1+\kappa^{1/q})n^{-r+d\left(\frac1p-\frac1q\right)_+}\|f\|_{X_{p}^r(\Bbb M)},
 $$
and for $q=\infty$,
$$
  \|f-L_{n,\infty}^{\Bbb M}(f)\|_{L_\infty(\Bbb M)}\le C(1+\kappa)n^{-r+d/p}\|f\|_{X_{p}^r(\Bbb M)},
$$
where $C>0$ is independent of $n$, $\kappa$, $f$, or $\mathcal{X}$.
\end{thm}

As a consequence of Theorem \ref{rem1.1}, we will obtain least squares quadrature errors.

\begin{thm}\label{thm1.2}  Let $2\le p\le\infty$, $0<\tau\le\infty$, and $r>d/2$. Suppose that $\mathcal{X}$
is an $L_{2}$-Marcinkiewicz-Zygmund family on $\Bbb M$ with associated weights $\tau$ and global condition number
$\kappa$, and $I_n^{\Bbb M}$ is  the  least squares quadrature rule induced by $\mathcal{X}$.
Then for any $f\in X_p^r(\Bbb M)$,
$$
  \left|\int_{\Bbb M}f({\bf x}){\rm d}\mu({\bf x})-I_n^{\Bbb M}(f)\right|\le C(1+\kappa^{1/2})n^{-r}\|f\|_{X_p^r(\Bbb M)},
$$
where $C>0$  is independent of $n$, $\kappa$, $f$, or $\mathcal{X}$.
\end{thm}

The remaining paper is structured as follows. We begin by
providing an overview of fundamental concepts related to compact
Riemannian manifolds in Subsect.2.1.  Then we delve into the
exploration of $L_p$-Marcinkiewicz-Zygmund families in
Subsect.2.2,  address the weighted least $\ell_p$ approximation
problems in Subsect.2.3, and  present  some auxiliary results in
Subsect.2.4, including   results of filtered polynomial
approximation. In Sect.\ref{sect3}, we give the detailed proofs of
Theorems 1.1-1.3.  Sect.\ref{sect4} is devoted to elaborating the
optimality of the obtained results by examining the (linear)
sampling numbers and optimal quadrature errors. Finally, in
Sect.\ref{sect5}, we conclude with a summary of our findings and
offer further discussions on potential directions for future
research.

\begin{nota}Throughout the paper, the notation $a_{n}\asymp b_{n}$ means $a_{n}\lesssim b_{n}$  and $a_{n}\gtrsim b_{n}$.
Here, $a_{n}\lesssim b_{n}\,(a_{n}\gtrsim b_{n})$ means that there exists a constant $c>0$ independent of $n$ such that $a_{n}\leq c b_{n}\,(b_{n}\leq c a_{n})$.
\end{nota}

\section{Preliminaries}\label{sect2}

In this section, we will provide  necessary materials required for
the remainder of the paper. We will begin by providing an overview
of fundamental concepts pertaining to Riemannian manifolds in a
general sense. Subsequently, we will introduce the concepts of
weighted least $\ell_p$ approximation and least squares quadrature
rules. We will also present necessary knowledge and establish key
lemmas.

\subsection{Compact Riemannian manifolds}\

Let $\Bbb M$ be a compact connected smooth $d$-dimensional
Riemannian manifold without boundary equipped with the normalized
Riemannian measure $\mu$. Typical examples are compact homogeneous
manifolds, which is of the form $G/H$. Here $G$ is a compact Lie
group and $H$ is its closed subgroup. For example, the $d$-torus
$\Bbb T^d$, $d=1,2,\dots$; the spheres $\Bbb S^d$,
$d=1,2,3,\dots$; the real projective spaces ${\rm P}^d(\Bbb R),
d=2,3,\dots$; the complex projective spaces ${\rm P}^{d}(\Bbb C),
d=4,6,8,\dots$ the  quaternion projective spaces ${\rm
P}^{2d}(\Bbb H),d=4,6,8,\dots$; the Cayley elliptic plane ${\rm
P}^{16}(Cay)$. See \cite{C,Ga,H1,H2} for details.

Let ${\rm dist}_{\Bbb M}({\bf x},{\bf y})$ be the geodesic
distance of points ${\bf x}$ and ${\bf y}$ on $\Bbb M$, and let
${\rm B}({\bf x},r)$ denote the ball of radius $r>0$ centred at
${\bf x}\in\Bbb M$. Then $(\Bbb M,\mu)$ is \emph{Ahlfors
$d$-regular} which means that there exists a constant
$C_{d}>0$ depending only on  the
dimension $d$ such that
$$C_{d}^{-1}\,r^d\le \mu\left({\rm B}({\bf x},r)\right)\le C_{d}\,r^d,\ {\rm for\ all}\ {\bf x}\in\Bbb M\ {\rm and}\ 0<r\le{\rm diam}(\Bbb M),$$
where $\mu(E)=\int_E{\rm d}\mu({\bf x})$ for a measurable subset $E$ of $\Bbb M$.

For $1\le p<\infty$, let $L_p(\Bbb M)$ be the set of all complex
valued measurable function defined on $\Bbb M$ endowed with
finite norm
$$\|f\|_{L_p(\Bbb M)}:=\left(\int_{\Bbb M}|f({\bf x})|^p{\rm d}\mu({\bf x})\right)^{1/p},$$
and let  $L_\infty(\Bbb M)\equiv C(\Bbb M)$ be the set of all complex valued continuous functions defined on $\Bbb M$ endowed with uniform norm
$$\|f\|_{L_\infty(\Bbb M)}:=\sup\limits_{{\bf x}\in\Bbb M}|f({\bf x})|.$$
 In particular, $L_2(\Bbb M)$ is a Hilbert space with inner product $$\langle
f,g\rangle_{L_2(\Bbb M)}:=\int_{\Bbb M}f({\bf x})\overline{g({\bf x})}{\rm d}\mu({\bf x}),\ {\rm for}\ f,g\in
L_2(\Bbb M),$$where $\overline{g}$ is the complex conjugate to $g$.

Let $\Delta_{\Bbb M}$ be the Laplace-Beltrami operator on $\Bbb M$, which is self-adjoint in $L_2(\Bbb M)$. This operator has a sequence of
eigenvalues $\{-\lambda_k^2\}_{k=0}^\infty$ arranged by
$$0=\lambda_0\le\lambda_1\le\lambda_2\le\dots\le\lambda_k\le\dots\to\infty,$$
and a complete orthonormal  system of eigenfunctions
$\{\phi_k\}_{k=0}^\infty$. Without loss of generality, we choose
each $\phi_k$ to be real-valued and $\phi_0=1$. Then
$\{\phi_k\}_{k=0}^\infty$  can form an orthonormal basis for
$L_2(\Bbb M)$.
 For each $n=0,1,2,\dots$,
with $\mathcal{P}_n$ we denote the set of ``polynomials" of degree
at most $ n$ on $\Bbb M$ by
$$\mathcal{P}_n:={\rm span}\left\{\phi_k:\,\lambda_k\le n\right\}.$$
This space  is called the \emph{diffusion polynomial space of
degree $n$ }on $\Bbb M$, and an element of $\mathcal{P}_n$ is
called a \emph{diffusion polynomial of degree at most $n$}. In the
following, we will refer to diffusion polynomials simply as
polynomials.

For $f\in L_p(\Bbb M)$, $1\le p\le \infty$, the formal Fourier
coefficients of $f$ are
$$\widehat{f}(k):=\int_{\Bbb M}f({\bf x}){\phi_k({\bf x})}{\rm d}\mu({\bf x}),\ k=0,1,2,\dots.$$
Then the \emph{Sobolev space} $H_p^r(\Bbb M)$, $r>0$ is defined by
$$H_p^r(\Bbb M):=\left\{f\in L_p(\Bbb M):\,(I-\Delta_{\Bbb M})^{r/2}f\in L_p(\Bbb M)\right\},$$with  Sobolev norm
$$\|f\|_{H_p^r(\Bbb M)}:=\left\|(I-\Delta_{\Bbb M})^{r/2}f\right\|_{L_p(\Bbb M)}=\left\|\sum_{k=0}^\infty(1+\lambda_k^2)^{r/2}\widehat{f}(k)\phi_k\right\|_{L_p(\Bbb M)},$$
where $I$ is the identity on $L_p(\Bbb M)$. More about the Sobolev
spaces $H_p^r(\Bbb M)$ may be found in the paper  \cite{BCCGST,
KS}.

\begin{rem}In particular, when $p=2$, it reduces to $H^r(\Bbb M)\equiv H_2^r(\Bbb M)$ with the Sobolev norm
$$\|f\|_{H^r(\Bbb M)}:=\left(\sum_{k=0}^\infty(1+\lambda_k^2)^r|\widehat{f}(k)|^2\right)^{1/2},$$as a special case in \cite[Section 3]{G}.
Under the conditions in this paper, \emph{two-sides Wely's law} always holds with exponent $d$ and constants $C>c>0$, i.e.,
$$
c\,n^d\le\sum_{k:\,\lambda_k\le n}|\phi_k({\bf x})|^2\le C\,n^d,
$$
which, by integrating with respect to ${\bf x}$, implies that
$$\dim\mathcal{P}_n:=\#\{\phi_k:\,\lambda_k\le n\}\asymp n^d,$$
and the Sobolev space $H^r(\Bbb M)$ is continuously embedded into
$C(\Bbb M)$ if and only if  $r>d/2$ (see e.g. \cite[Proposition
4.2]{G}).
\end{rem}

Let
$$E_n(f)_{p}:=\inf_{g\in\mathcal{P}_n}\|f-g\|_{L_{p}(\Bbb M)}
$$ denote the \emph{best approximation} of  $f$ in $L_p(\Bbb M)$ by $\mathcal P_n$.
Given $1\le p\le \infty$, $r>0$, and $0<\tau\le\infty$, we can
define the \emph{Besov space} $B_{p,\tau}^r(\Bbb M)$ to be the
space of all real functions $f$ with finite quasi-norm
$$\|f\|_{B_{p,\tau}^r(\Bbb M)}:=\|f\|_{L_{p}(\Bbb M)}+\left(\sum\limits_{j=0}^\infty2^{jr\tau}E_{2^j}(f)_{p}^\tau\right)^{1/\tau}.$$
There are other equivalent definitions of  the Besov spaces that
are equivalent (see e.g. \cite{GP,DI,P}).

For $1\le p\le\infty$, the Sobolev space $H_p^r(\Bbb M)$ is
continuously embedded into $C(\Bbb M)$ if and only if $r>d/p$ (see
e.g. \cite[Section 2.7]{A} or \cite{BCCGST}). Meanwhile, the Besov
space $B_{p,\tau}^r(\Bbb M)$ is also continuously embedded into
$C(\Bbb M)$ if $r>d/p$. These embedding relations  guarantee that
any function in the above Sobolev and Besov spacea has a
representation by a continuous function. In the sequel, we denote
 $X_p^r(\Bbb M)$  to be the  Sobolev space $H_p^r(\Bbb M)$ or the
Besov space
 $B_{p,\tau}^r(\Bbb M)$.

\subsection{$L_p$-Marcinkiewicz-Zygmund families}\

Now we give the precise definition of Marcinkiewicz-Zygmund families in $L_p$.

\begin{defn} Let $\mathcal{X}=\{{\bf x}_{n,k}:\, k=1,2,\dots,N_n,\,n=1,2,\dots\}$
be a doubly-indexed set of points in $\Bbb M$, and
 $\tau=\{\tau_{n,k}:\,k=1,2,\dots,N_n,\,n=1,2,\dots\}$ be a family of positive weights.  Then

(i) for $1\le  p<\infty$, $\mathcal{X}$ is
 called an $L_p$-Marcinkiewicz-Zygmund family on $\Bbb M$, denoted by $L_p$-MZ, if there exist constants $A,\, B >0$ independent
of $n$ such that
\begin{equation}\label{2.2}
  A\|Q\|_{L_p(\Bbb M)}^p\le\sum\limits_{k=1}^{N_n}\tau_{n,k}|Q({\bf x}_{n,k})|^p\le B\|Q\|_{L_p(\Bbb M)}^p,\ {\rm for\ all}\ Q\in \mathcal{P}_n.
\end{equation}
\ \ \ (ii) for $p=\infty$, $\mathcal{X}$ is called an
$L_{\infty}$-Marcinkiewicz-Zygmund family on $\Bbb M$, denoted by
$L_{\infty}$-MZ, if there exist constants $A>0$ independent of $n$
such that
\begin{equation}\label{2.3}
  A\|Q\|_{L_\infty(\Bbb M)}\le\max\limits_{1\le k\le N_n}|Q({\bf x}_{n,k})|\le \|Q\|_{L_\infty(\Bbb M)},\ {\rm for\ all}\ Q\in \mathcal{P}_n.
\end{equation}

The ratios $\kappa=B/A$ for $1\le p<\infty$ and $\kappa=1/A$ if
$p=\infty$ are the {global condition numbers} of the
$L_p$-Marcinkiewicz-Zygmund  family, and $\mathcal{X}_n:=\{{\bf
x}_{n,k}:\,k=1,2,\dots,N_n\}$ is the $n$-th layer of
$\mathcal{X}$.
\end{defn}

Let $\mathcal{X}$ be an $L_{p}$-Marcinkiewicz-Zygmund family with associated
weights $\tau$ on $\Bbb M$. We set for $1\le p<\infty$,
$$
  \mu_n:=\sum\limits_{k=1}^{N_n}\tau_{n,k}\delta_{{\bf x}_{n,k}},
$$where $\delta_{\bf z}(g)=g({\bf z})$ for a function $g$ is the evaluation operator.
For any $f\in C(\Bbb M)$, we define for $1\le p<\infty$,
$$\|f\|_{(p)}:=\left(\int_{\Bbb M}|f({\bf x})|^p{\rm d}\mu_n\right)^{1/p}=\left(\sum\limits_{k=1}^{N_n}\tau_{n,k}|f({\bf x}_{n,k})|^p\right)^{1/p},$$
and for  $p=\infty$,  $$\|f\|_{(\infty)}:=\max\limits_{1\le k\le
N_n}|f({\bf x}_{n,k})|.$$ Hence, for $1\le p\le \infty$,
$(\mathcal{P}_n,\|\cdot\|_{(p)})$ is a Banach space.  Using the
inequalities  \eqref{2.2} and \eqref{2.3}, we can conclude that
the $L_{p}$ norm of a polynomial of degree at most $n$ on $\Bbb M$
is comparable to its discrete version, which is given by the
weighted $ \ell_p$ norm of its restriction to $ \mathcal{X}_n $.
Moreover,  if $ Q\in \mathcal{P}_n $ and $ Q({\bf x}_{n,k}) = 0 $
for $ k = 1,\dots,N_n $, then $ Q = 0 $, which means that $
\mathcal X_n $ contains more than $ \dim\mathcal{P}_n $ points.
Therefore, $ \mathcal X_n $ is not an interpolating set for $
\mathcal{P}_n $.

\vskip2mm

Filbir and Mhaskar in \cite{FM} established the existence of
$L_p$-MZ families on compact smooth connected Riemanian manifold
$\mathbb{M}$.
 Before we proceed to state this result, we introduce the concept of a maximal separated subset on $\mathbb{M}$.

Let $\Xi$ be a finite subset of $ \mathbb{M}$. We say that $\Xi$
is \emph{$\varepsilon$-separated} if the minimum distance between
any two distinct points $\boldsymbol{\xi}$ and $\boldsymbol{\eta}$
in $\Xi$ is at least $\varepsilon$. In other words, we require
$$q(\Xi):=\min\limits_{\substack{\boldsymbol{\xi},\boldsymbol{\eta}\in\Xi
\\ \boldsymbol{\xi}\neq\boldsymbol{\eta}}}{\rm dist}_{\mathbb{M}}(\boldsymbol{\xi},\boldsymbol{\eta})\geq\varepsilon.$$
Furthermore, if an $\varepsilon$-separated subset $\Xi$ of
$\mathbb{M}$ satisfies
$$\delta(\Xi):=\max\limits_{\boldsymbol{\xi}\in\mathbb{M}}{\rm
dist}_{\mathbb{M}}(\boldsymbol{\xi},\Xi)\leq\varepsilon,$$then we
refer to it as a \emph{maximal $\varepsilon$-separated} subset of
$\mathbb{M}$, where ${\rm dist}_{\Bbb M}(\cdot,\Xi):\,\Bbb
M\to[0,\infty)$ is the distance function of $\Xi$ with respect to
${\rm dist}_{\Bbb M}$  defined by
$${\rm dist}_{\Bbb M}(\boldsymbol{\xi},\Xi):=\inf_{\boldsymbol\eta\in \Xi}{\rm dist}_{\Bbb M}(\boldsymbol{\xi},\boldsymbol\eta).$$
 In simpler terms, a maximal $\varepsilon$-separated subset $\Xi$ is
 an $\varepsilon$-separated subset of $\mathbb{M}$ such that every point in $\mathbb{M}$
 lies within a distance of $\varepsilon$ from at least one point in $\Xi$.
  We can express this as $$\mathbb{M}=\bigcup\limits_{\boldsymbol{\xi}\in\Xi}{\rm B}(\boldsymbol{\xi},\varepsilon).$$
  It is worth noting that if $\Xi$ is a maximal $\varepsilon$-separated subset of $\mathbb{M}$,
   then there exists a constant $C_d>0$ that solely depends on the dimension $d$,
    such that for all ${\bf x}\in\mathbb{M}$, we have
    \begin{equation}\label{2.4-0}
    1\leq \sum_{\boldsymbol{\xi}\in\Xi}\chi_{{\rm B}(\boldsymbol{\xi},\varepsilon)}({\bf x})\leq C_d.
    \end{equation}
    We can conclude from Ahlfors $d$-regularity that $\# \Xi \asymp \varepsilon^{-d}$,
    which essentially means that the cardinality of $\Xi$ is proportional to the reciprocal of the $\varepsilon$ raised to the power of the dimension $d$.

 Filbir and Mhaskar  gave the existence of
$L_p$-MZ families on $\mathbb{M}$ in \cite[Theorem 5.1]{FM}). For
our purpose we  restated it as  follows.

\begin{lem}{\rm (\cite[Theorem 5.1]{FM}).}\label{lem4.5} Let $1\le p\le\infty$.
Then there exists a constant $\delta_0>0$ depending only on $d$ and $p$ such that
 given a maximal $\delta/n$-separated subset $\Xi$ of $\Bbb M$ with $\delta\in(0,\delta_0]$,
 there exists a set of positive numbers $\{\lambda_{\boldsymbol\xi}\}_{ \boldsymbol\xi\in\Xi}$ such that for all  $f\in\mathcal{P}_n$,
$$\|f\|_{L_p(\Bbb M)}\asymp \Bigg\{\begin{aligned}&\Big(\sum_{\boldsymbol\xi\in\Xi}\lambda_{\boldsymbol\xi}|f(\boldsymbol\xi)|^p\Big)^{1/p},  &&1\le p<\infty,
  \\&\max\limits_{\boldsymbol\xi\in\Xi}|f(\boldsymbol\xi)|, &&p=\infty,\end{aligned},$$
  and $\lambda_{\boldsymbol\xi}\asymp \mu \left({\rm B}(\boldsymbol\xi,\delta/n)\right)\asymp n^{-d}$
  for all $\boldsymbol\xi\in\Xi$, where the constants of equivalence depend only on $d$ and $p$.
\end{lem}

\subsection{Weighted least $\ell_p$ approximation and least squares
quadrature.
}\

For $1\le p\le\infty$, let $\mathcal{X}$ be an $L_p$-MZ family with associated weights $\tau$. We want to
 approximate $f$ using only these sampling nodes. In line with the method of Gr\"ochenig  in \cite{G},
  this can be done by solving the  weighted least $\ell_p$ approximation
problem with samples taken from the $n$-th layer $\mathcal{X}_n$
of $\mathcal{X}$. Therefore, we can define
  the
\emph{weighted least $\ell_p$ approximation} of $f\in C({\Bbb M})$
by
$$
          L_{n,p}^{\Bbb M}(f):=\arg\min\limits_{g\in\mathcal{P}_n}\,\left(\sum\limits_{k=1}^{N_n}\tau_{n,k}|f({\bf x}_{n,k})-g({\bf x}_{n,k})|^p\right)^{1/p}.
$$

Obviously, the solution of this problem exists for all $f\in
C(\Bbb M)$. Further, if $1< p<\infty$, then the solution is
unique. If $p=1$ or $\infty$, then the solution may be not unique
and we can choose any solution to be $ L_{n,p}^{\Bbb M}(f)$. It
can be seen that $L_{n,p}^{\Bbb M}$ is nonlinear except for the
case $p=2$.
 That is, for $p\neq2$, there exist
$f_1,f_2\in C({\Bbb M})$ such that
$$L_{n,p}^{\Bbb M}(f_1+f_2)\neq L_{n,p}^{\Bbb M}(f_1)+L_{n,p}^{\Bbb M}(f_2).$$ However,
 for  $f\in C({\Bbb M})$ and $Q\in\mathcal{P}_n$,  by the definition of $L_{n,p}^{\Bbb M}(f)$ we have
\begin{align}\label{2.4}
 \|f-L_{n,p}^{\Bbb M}(f)\|_{(p)}\notag &= \|(f-Q)-L_{n,p}^{\Bbb M}(f-Q)\|_{(p)}\notag\\&\le \|f-Q\|_{(p)}.
\end{align}This is a critical property. Besides, we have
\begin{equation}\label{2.5}
  L_{n,p}^{\Bbb M}(Q)=Q,\ {\rm for\ all}\ Q\in\mathcal{P}_n.
\end{equation}

Now we turn to consider the most interesting  case $p=2$ in which
$L_{n}^{\Bbb M}\equiv L_{n,2}^{\Bbb M}$ has many good properties
(see \cite{LW}). Indeed, $L_{n}^{\Bbb M}$ is a bounded linear
operator on $C({\Bbb M})$ satisfying that $(L_{n}^{\Bbb
M})^2=L_{n}^{\Bbb M}$, and the range of $L_{n}^{\Bbb M}$ is
$\mathcal{P}_n$. This means that $L_{n}^{\Bbb M}$ is a projection
on $\mathcal{P}_n$. If we define the \emph{weighted discretized
inner product} on $C({\Bbb M})$ by
$$\langle
f,g\rangle_{(2)}:=\sum\limits_{k=1}^{N_n}\tau_{n,k}f({\bf x}_{n,k})g({\bf x}_{n,k}),$$
then $L_{n}^{\Bbb M}$ is just the orthogonal projection onto $\mathcal{P}_n$
with respect to the discretized inner product $\langle
\cdot,\cdot\rangle_{(2)}$. Let $\varphi_k,k=1,\dots,\dim\mathcal{P}_n$ be  an orthonormal basis for $\mathcal{P}_n$
with respect to the weighted discrete scalar product $\langle\cdot,\cdot\rangle_{(2)}$.
Such $\varphi_k,k=1,\dots,{\rm dim}\mathcal{P}_n$ can be obtained by applying Gram-Schmidt orthogonalization
process to the basis $\{\phi_k:k=0,\dots,\dim\mathcal{P}_n-1\}$ of $\mathcal{P}_n$. We set
$$D_n({\bf x},{\bf y}):=\sum_{k=1}^{\dim\mathcal{P}_n}\varphi_k({\bf x})\varphi_k({\bf y}).$$
Then $D_n({\bf x},{\bf y})$ is the reproducing kernel of $\mathcal{P}_n$
with respect to the weighted discretized inner product $\langle\cdot,\cdot\rangle_{(2)}$ satisfying the following properties:

(i) For any ${\bf x},{\bf y}\in\Bbb M$, $D_n({\bf x},{\bf y})=D_n({\bf y},{\bf x})$;

(ii) For any fixed ${\bf y}\in\Bbb M$, $D_n(\cdot,{\bf y})\in\mathcal{P}_n$;

(iii) For any $Q\in\mathcal{P}_n$ and ${\bf x}\in\Bbb M$,
$$Q({\bf x})=\langle Q,D_n({\bf x},\cdot)\rangle_{(2)}=\sum_{k=1}^{N_n}\tau_{n,k}Q({\bf x}_{n,k})D_n({\bf x},{\bf x}_{n,k}).$$
Therefore,  for $f\in C({\Bbb M})$,
\begin{equation*}
  L_{n}^{\Bbb M}(f)({\bf x})=\langle f,
  D_n({\bf x},\cdot)\rangle_{(2)}=\sum_{k=1}^{N_n}\tau_{n,k}f({\bf x}_{n,k})D_n({\bf x},{\bf x}_{n,k}).
\end{equation*}
We call  $L_{n}^{\Bbb M}(f)$   the \emph{weighted least squares polynomial} and
$L_{n}^{\Bbb M}$ the \emph{weighted least squares operator}.

We also want to approximate the integral
  $\int_{\Bbb M} f{\rm d}\mu$ using only finite samples. To this end, Gr\"ochenig  in \cite{G} used the frame theory
  to construct   \emph{least squares  quadrature rules}
$$
I_n^{\Bbb M}(f)=\sum_{k=1}^{N_n}w_{n,k}f({\bf x}_{n,k})
$$ induced by an $L_{2}$-MZ
  family.  Lu and Wang in \cite{LW} proved that
$$w_{n,k}=\tau_{n,k}\int_{\Bbb M}D_n({\bf x},{\bf x}_{n,k}){\rm d}\mu({\bf x}), \ \ \ I_n^{\Bbb M}(f)=\int_{\Bbb M}L_{n}^{\Bbb M}(f)({\bf x}){\rm d}\mu({\bf x}), $$
and
\begin{equation}\label{2.7}\left|\int_{\Bbb M}f({\bf x}){\rm d}\mu({\bf x})-I_n^{\Bbb M}(f)\right|\le \left\|f-L_{n}^{\Bbb M}(f)\right\|_{L_2(\Bbb M)}.\end{equation}
\subsection{Auxiliary lemmas}
\

First we present the \emph{Nikolskii  inequality} on $\Bbb M$  proved in \cite{Mh1}.

\begin{prop}{\rm(\cite[Lemma 5.5]{Mh1}).}\label{prop2.1} Let $1\le p,q\le \infty$ and $Q\in \mathcal{P}_n$. Then
$$
\|Q\|_{L_{q}(\Bbb M)}\lesssim n^{d\left(\frac1p-\frac1q\right)_+}\|Q\|_{L_{p}(\Bbb M)}.
$$
\end{prop}

Now we recall some knowledge about  filtered polynomial
approximation on  $\Bbb M$. In the sequel, we always assume that
$\eta \in C^\infty[0,\infty)$ is a ``$C^\infty$-filter'' such that
$\eta(t)=1$ if $t\in[0,1]$ and $\eta(t)=0$ if $t\ge2$. We will fix
$\eta$ for the rest of the paper, so that the dependence of
different constants on the choice of $\eta$ will not be mentioned.

For any positive integer $n$ and $\x,\y\in \Bbb M$,  we write
$$K_{n}({\bf x},{\bf y})\equiv K_{n,\eta}({\bf x},{\bf y})
=\sum_{k=0}^\infty \eta\left( \frac{\lambda_k}n\right)\phi_k({\bf x}){\phi_k({\bf y})}.$$
Then, $K_n({\bf x},{\bf y})$ satisfies the following properties (see e.g. \cite[Theorem 4.1]{MM} and \cite[Theorem 2.1]{FM1}):

(i) for any ${\bf x},{\bf y}\in\Bbb M$, $K_n({\bf x},{\bf y})=K_n({\bf y},{\bf x})$;

(ii) for any fixed ${\bf y}\in\Bbb M$, $K_n(\cdot,{\bf y})\in\mathcal{P}_{2n}$;

(iii) for any $Q\in\mathcal{P}_n$ and ${\bf x}\in\Bbb M$,
$$Q({\bf x})=\langle Q,K_n({\bf x},\cdot)\rangle_{L_2(\Bbb M)}=\int_{\Bbb M}Q({\bf y})K_n({\bf x},{\bf y}){\rm d}\mu({\bf y});$$

(iv) for any $\gamma>d$ and ${\bf x},{\bf y}\in\Bbb M$,
$$
 |K_{n}({\bf x},{\bf y})|\le \frac{cn^d}{\left(1+n\,{\rm dist}_{\Bbb M}({\bf x},{\bf y})\right)^\gamma};
$$

(v) for any ${\bf x}\in\Bbb M$,
$$
  \|K_{n}({\bf x},\cdot)\|_{L_1(\Bbb M)}\le C,
$$where $C>0$ depends only on $d$.

For $f\in L_1(\Bbb M)$, we define the \emph{filtered polynomial
approximation operator} $V_n$  by
$$
V_{n}(f)({\bf x}):= \langle f,
K_{n}({\bf x},\cdot)\rangle_{L_2(\Bbb M)}=\sum_{k:\,\lambda_k\le 2n} \eta\left( \frac{\lambda_k}n\right)\widehat{f}(k)\phi_k({\bf x}),
$$
Then, $V_n$ is a linear operator and satisfies the following properties (see e.g. \cite{MW}):

(i) for all $f\in L_1(\Bbb M)$, $V_n(f)\in\mathcal{P}_{2n}$;

(ii) for all $Q\in\mathcal{P}_n$, $Q=V_n(Q)$;

(iii) for all $1< p<\infty$,
\begin{equation}\label{2.10}
  \|V_n\|_{(p,p)}\le \|V_n\|_{(1,1)}=\|V_n\|_{(\infty,\infty)}=\max_{{\bf x}\in\Bbb M}\|K_{n}({\bf x},\cdot)\|_{L_1(\Bbb M)}\le C,
\end{equation}
where $C>0$ depends only on $d$, and $$\|V_n\|_{(p,p)}:=\sup_{\|f\|_{L_p(\Bbb M)}\le 1}\|V_n(f)\|_{L_p(\Bbb M)};$$

(iv) for all $1\le p\le\infty$,
\begin{equation}\label{2.11}
  \|f-V_n(f)\|_{L_{p}(\Bbb M)}\le (1+C)E_n(f)_{p}.
\end{equation}

It follows from \cite[Lemma 5.4]{Mh1} and the definition of the
Besov space $B_{p,\tau}^r(\Bbb M)$ that the Jackson inequality
\begin{equation}\label{2.12}
  E_n(f)_{p}\lesssim n^{-r}\|f\|_{X_{p}^{r}(\Bbb M)},
\end{equation}holds for $f\in X_{p}^r(\Bbb M)$, $1\le
p\le\infty$,  $0<\tau\le \infty$, and $r>0$, which,  combining with
\eqref{2.11}, we obtain
\begin{equation}\label{2.8}
  \|f-V_n(f)\|_{L_{p}(\Bbb M)}\lesssim n^{-r}\|f\|_{X_{p}^{r}(\Bbb M)}.
  \end{equation}This gives the convergence order for the filtered polynomial
approximation of a smooth function from $X_p^r(\Bbb M)$ measured
in the $L_p(\Bbb M)$ norm. We extend this estimate to more general
case as stated below.

\begin{prop}\label{prop2.3}Let $1\le p,q\le\infty$, $0<\tau\le\infty$, and $r>d(1/p-1/q)_+$. Then for
$f\in X_p^r(\Bbb M)$, we have
\begin{equation}\label{2.13}
  \|f-V_n(f)\|_{L_q(\Bbb M)}\lesssim n^{-r+d\left(\frac1p-\frac1q\right)_+}\|f\|_{X_{p}^{r}(\Bbb M)}.
\end{equation}
\end{prop}
\begin{proof}
For any $f\in X_p^r(\Bbb M)$, we  define
$$\sigma_1(f)=V_{1}(f), \ \ \sigma_j(f)=V_{2^{j-1}}(f)-V_{2^{j-2}}(f),\ \
{\rm for}\ j\ge2.$$Then $\sigma_j(f)\in\mathcal{P}_{2^{j}}$ and
$$f=\sum_{j=1}^\infty\sigma_j(f)$$ holds in the uniform norm. Meanwhile, \eqref{2.8} implies
\begin{align*}
  \|\sigma_j(f)\|_{L_{p}(\Bbb M)}&\le \|f-V_{2^{j-1}}(f)\|_{L_{p}(\Bbb M)}+\|f-V_{2^{j-2}}(f)\|_{L_{p}(\Bbb M)}
  \\&\lesssim 2^{-jr}\|f\|_{X_{p}^r(\Bbb M)}.
\end{align*}Therefore, applying Proposition \ref{prop2.1}, we deduce that
\begin{align}\label{2.14}
  \|\sigma_j(f)\|_{L_{q}(\Bbb M)}&\le 2^{jd\left(\frac1p-\frac1q\right)_+}\|\sigma_j(f)\|_{L_{p}(\Bbb M)}\notag
  \\&\lesssim \left(2^{-j}\right)^{r-d\left(\frac1p-\frac1q\right)_+}\|f\|_{X_{p}^r(\Bbb M)}.
\end{align}

 Given a positive integer $n$, choose a integer $m$ such that $2^{m}\le n<2^{m+1}$.
We recall that for all $f\in X_p^r(\Bbb M)$, the series
$\sum_{j=1}^\infty\sigma_j(f)$ converges to $f$ in the uniform
norm.
 Thus, since $V_n(Q)=Q$ for all $Q\in\mathcal{P}_n$, we obtain
$$
  f-V_n(f)=\sum_{j=m+1}^\infty\left(\sigma_j(f)-V_n(\sigma_j(f))\right),
$$which leads to
\begin{align*}
  \|f-V_n(f)\|_{L_q(\Bbb M)}&\le\sum_{j=m+1}^\infty\left\{\|\sigma_j(f)\|_{L_q(\Bbb M)}+\|V_n(\sigma_j(f))\|_{L_q(\Bbb M)}\right\}.
\end{align*}
Then, by \eqref{2.10} and \eqref{2.14}, this implies
\begin{align*}
  \|f-V_n(f)\|_{L_q(\Bbb M)}&\lesssim\sum_{j=m+1}^\infty\|\sigma_j(f)\|_{L_q(\Bbb M)}
  \\&\lesssim\sum_{j=m+1}^\infty \left(2^{-j}\right)^{r-d\left(\frac1p-\frac1q\right)_+}\|f\|_{X_{p}^r(\Bbb M)}
  \\&\asymp n^{r-d\left(\frac1p-\frac1q\right)_+}\|f\|_{X_{p}^r(\Bbb M)},
\end{align*}which gives the estimates \eqref{2.13}.
\end{proof}
The  following lemma is critical in the proof of Theorem
\ref{thm1.1}, which gives  regularity  condition of a discrete
measure.

\begin{lem}{\rm (\cite[Theorem 5.5]{FM} and \cite[Lemma 7]{MW}).}\label{lem2.1}
Let $\Xi$ be a finite set of points in
$\Bbb M$, and let $\Lambda=\{\lambda_{\boldsymbol\xi
}\}_{{\boldsymbol\xi
}\in\Xi}$ be a finite set of positive numbers. If there exists $1\le p_0<\infty$ such that
\begin{equation}\label{2.15}
\sum_{\boldsymbol\xi\in\Xi}\lambda_{\boldsymbol\xi
}|Q({\boldsymbol\xi})|^{p_0}\le
C_1\int_{\Bbb M}|Q({\bf x})|^{p_0}{\rm d}\mu({\bf x}),
\end{equation}holds for all $Q\in \mathcal{P} _{n}$ with $C_1>0$,
then the following \textbf{regularity
condition }
\begin{equation}\label{2.16}\sum_{\boldsymbol\xi\in\Xi\cap{\rm B}\left({\bf y},\frac1n\right)}\lambda_{\boldsymbol\xi
}
 \le C_1C_2\int_{{\rm B}\left({\bf y},\frac1n\right)}{\rm d}\mu({\bf x}),\ {\rm for\ all}\ {\bf y}\in\Bbb M.\end{equation} holds,
 where $C_2>0$ depends only on $d$ and  $p_0$.

Conversely, if  the regularity condition \eqref{2.16} holds,
then for any positive integer $m\geq n$ and $1\le p<\infty$ we have
$$\sum_{\boldsymbol\xi\in\Xi}\lambda_{\boldsymbol\xi
}|Q({\boldsymbol\xi})|^{p}\le
C_1C_2C_3\left(\frac{m}{n}\right)^{d}\int_{\Bbb M}|Q({\bf x})|^{p}{\rm d}\mu({\bf x}),$$holds for all $Q\in \mathcal{P} _{m}$ with $C_3>0$ depending only on  $d$ and $p$.
\end{lem}
We remark that  Lemma \ref{lem2.1} is essentially the same as \cite[Lemma 7]{MW} although there is a little different of their form.

\section{Proofs of main results}\label{sect3}

Theorem \ref{thm1.2} follows from  Theorem \ref{rem1.1} and \eqref{2.7} straightly. Thus we only need to prove Theorems \ref{thm1.1} and \ref{rem1.1}.
For this purpose, we give the following lemma which implies the above theorems immediately.

\begin{lem}\label{thm3.3} Let $1\le p,q,t\le\infty$, $0<\tau\le\infty$, and $r>d\max\{1/p,1/t\}$. Suppose that $\mathcal{X}$ is an $L_{t}$-MZ
family on $\Bbb M$ with associated weights $\tau$ and global
condition number $\kappa$, and $L_{n,t}^{\Bbb M}$ is the
corresponding weighted least $\ell_t$ approximation. If $f\in
X_{p}^r(\Bbb M)$, then we have for $1\le t<\infty$,
\begin{equation}\label{3.1}
  \|f-L_{n,t}^{\Bbb M}(f)\|_{L_{q}(\Bbb M)}\le C(1+\kappa^{1/t})\,n^{-r+d\left[\left(\frac1p-\frac1t\right)_++\left(\frac1t-\frac1q\right)_+\right]}\|f\|_{X_{p}^r(\Bbb M)},
\end{equation}
and for $t=\infty$,
\begin{equation}\label{3.2}
  \|f-L_{n,\infty}^{\Bbb M}(f)\|_{L_{q}(\Bbb M)}\le C(1+\kappa)\,n^{-r+d/p}\|f\|_{X_{p}^r(\Bbb M)},
\end{equation}where $C>0$ is independent of $n$, $\kappa$, $f$, or
$\mathcal{X}$. Specifically, if $r>d\max\{1/p,1/t\}$ and $t$ is
between $p$ and $q$,  then we have
\begin{equation*}
  \|f-L_{n,t}^{\Bbb M}(f)\|_{L_{q}(\Bbb M)}\le C(1+\kappa^{1/t})\,n^{-r+d(1/p-1/q)_+}\|f\|_{X_{p}^r(\Bbb
  M)}.
\end{equation*}
\end{lem}
\begin{proof}
We keep the notations $V_j$ and $\sigma_j$ as in the proof of
Proposition \ref{prop2.3}. For a positive integer $n$, let  $s$ be
an integer such that $2^s\le n<2^{s+1}$. Observe that for any
$f\in X_{p}^r(\Bbb M)$ with $r>d/p$,
\begin{align}\label{3.3}
  \|f-L_{n,t}^{\Bbb M}(f)\|_{L_{q}(\Bbb M)}&\le
  \|f-V_{2^{s-1}}(f)\|_{L_{q}(\Bbb M)}+\|L_{n,t}^{\Bbb M}(f)-V_{2^{s-1}}(f)\|_{L_{q}(\Bbb M)}\notag
  \\&=:{\rm I}+{\rm II}.
\end{align}
To estimate  ${\rm I}=\|f-V_{2^{s-1}}(f)\|_{L_{q}(\Bbb M)}$, we
can use \eqref{2.13} to obtain
\begin{align}\label{3.4}
  {\rm I}\ &\lesssim   \left(2^{s-1}\right)^{-r+d\left(\frac1p-\frac1q\right)_+}\|f\|_{X_{p}^r(\Bbb M)}\notag
  \\&\lesssim n^{-r+d\left(\frac1p-\frac1q\right)_+}\|f\|_{X_{p}^r(\Bbb M)}.
\end{align}
To estimate $\text{II} = \|L_{n,t}^{\Bbb
M}(f)-V_{2^{s-1}}(f)\|_{L_q(\Bbb M)}$, we can use Proposition
\ref{prop2.1}, \eqref{2.2}, and \eqref{2.4} to obtain that for
$1\le t<\infty$,
\begin{align}
  {\rm II}\ &\lesssim n^{d(\frac1t-\frac1q)_+}\|L_{n,t}^{\Bbb M}(f)-V_{2^{s-1}}(f)\|_{L_{t}(\Bbb M)}
  \notag\\&\le A^{-1/t}n^{d(\frac1t-\frac1q)_+}\|L_{n,t}^{\Bbb M}(f)-V_{2^{s-1}}(f)\|_{(t)}
  \notag\\&\le A^{-1/t}n^{d(\frac1t-\frac1q)_+}\{\|L_{n,t}^{\Bbb M}(f)-f\|_{(t)}+\|f-V_{2^{s-1}}(f)\|_{(t)}\}
  \notag\\&\le 2 A^{-1/t}n^{d(\frac1t-\frac1q)_+}\|f-V_{2^{s-1}}(f)\|_{(t)},\label{3.5}
\end{align}where we use the fact that $L_{n,t}^{\Bbb M}(f)-V_{2^{s-1}}(f)\in \mathcal{P}_n$.
Since for all $f\in X_{p}^r(\Bbb M)$ with $r>d/p$, the series $\sum_{j=s+1}^\infty\sigma_j(f)$ converges to $f-V_{2^{s-1}}(f)$ in the uniform norm, we have
 \begin{equation}\label{3.6}
 \|f-V_{2^{s-1}}(f)\|_{(t)}\le \sum\limits_{j=s+1}^\infty\|
  \sigma_j(f)\|_{(t)}.
 \end{equation}
We need to estimate $\|\sigma_j(f)\|_{(t)}$ for $j \geq s+1$. Clearly, \eqref{2.2} implies that
\eqref{2.15} is valid for sampling points $\{{\bf x}_{n,k}\}_{k=1}^{N_n}$ with weights $\{\tau_{n,k}\}_{k=1}^{N_n}$, $p_0=t$, and $C_1=B$.
We have already observed that $\sigma_j(f) \in \mathcal{P}_{2^j}$
and $2^j \geq n$ for $j \geq s+1$. Therefore, based on Lemma \ref{lem2.1}, Proposition \ref{prop2.1}, and \eqref{2.14}, we can conclude that
 \begin{align*}
\| \sigma_j(f) \|_{(t)}^t &= \sum_{k=1}^{N_n} \tau_{n,k} | \sigma_j(f)(\mathbf{x}_{n,k})|^t
\\&\lesssim B\left(\frac{2^j}{n}\right)^d \| \sigma_j(f) \|_{L_t(\mathbb{M})}^t \\
&\lesssim B\left(\frac{2^j}{n}\right)^d (2^j)^{td\left(\frac{1}{p}-\frac{1}{t}\right)_+} \| \sigma_j(f) \|_{L_p(\mathbb{M})}^t \\
&\lesssim B\left(\frac{2^j}{n}\right)^d (2^j)^{-rt+td\left(\frac{1}{p}-\frac{1}{t}\right)_+} \|f\|_{X_p^r(\mathbb{M})}^t \\
&= Bn^{-d}(2^j)^{-t\left\{r-d\left[\left(\frac{1}{p}-\frac{1}{t}\right)_++\frac{1}{t}\right]\right\}} \|f\|_{X_p^r(\mathbb{M})}^t,
\end{align*}i.e.,
\begin{equation}\label{3.7}
  \| \sigma_j(f)\|_{(t)}\lesssim B^{1/t}n^{-d/t}(2^{j})^{-r+d\,\left[\left(\frac1p-\frac1t\right)_++\frac1t\right]}\|f\|_{X_{p}^r(\Bbb M)}.
\end{equation}
Hence, for $r>d\max\{1/p,1/t\}$, we obtain using \eqref{3.5}, \eqref{3.6}, and
\eqref{3.7} that
\begin{align}\label{3.8}
  &{\rm II}\ \lesssim A^{-1/t}n^{d\,(\frac1t-\frac1q)_+}\sum\limits_{j=s+1}^\infty\| \sigma_j(f)\|_{(t)}\nonumber
  \\&\quad\lesssim \kappa^{1/t}n^{d\,(\frac1t-\frac1q)_+} n^{-d/t}
  \sum\limits_{j=s+1}^\infty (2^{j})^{-r+d\,\left[(\frac1p-\frac1t)_++\frac1t\right]}\|f\|_{X_{p}^r(\Bbb M)}\nonumber
  \\&\quad\lesssim\kappa^{1/t}n^{-r+d\,\left[(\frac1p-\frac1t)_++(\frac1t-\frac1q)_+\right]}\|f\|_{X_{p}^r(\Bbb M)}.
\end{align}
Combining  \eqref{3.3}, \eqref{3.4}, with \eqref{3.8}, we obtain
\eqref{3.1}.

For  $t=\infty$, by Proposition \ref{prop2.1} and \eqref{2.14} we
have  for $r>d/p$,
\begin{equation}\label{3.9}
  \| \sigma_j(f)\|_{(\infty)}\le \| \sigma_j(f)\|_{L_\infty(\Bbb M)}
  \lesssim (2^{j})^{d/p}\|\sigma_j (f)\|_{L_{p}(\Bbb M)}\lesssim (2^{j})^{-r+d/p}\|f\|_{X_{p}^r(\Bbb M)},
\end{equation}and
$$
\|L_{n,\infty}^{\Bbb M}(f)-V_{2^{s-1}}(f)\|_{(\infty)}\lesssim\sum\limits_{j=s+1}^\infty\|
  \sigma_j(f)\|_{(\infty)}.
$$
Together with \eqref{2.3} and \eqref{3.9}, this will lead to
\begin{align}\label{3.10}
  {\rm II}\,&\le\|L_{n,\infty}^{\Bbb M}(f)-V_{2^{s-1}}(f)\|_{L_\infty(\Bbb M)}\notag
      \\&\le \kappa\,\|L_{n,\infty}^{\Bbb M}(f)-V_{2^{s-1}}(f)\|_{(\infty)}\notag
  \\&\lesssim \kappa\,\sum\limits_{j=s+1}^\infty (2^{j})^{-r+d/p}\|f\|_{X_{p}^r(\Bbb M)}\notag
  \\&\lesssim \kappa\, n^{-r+d/p}\|f\|_{X_{p}^r(\Bbb M)}.
\end{align}Therefore, from \eqref{3.3}, \eqref{3.4}, and \eqref{3.10}, we get
\begin{align*}
  \|f-L_{n,\infty}^{\Bbb M}(f)\|_{L_{q}(\Bbb M)}\notag&\le
  \|f-V_{2^{s-1}}(f)\|_{L_{q}(\Bbb M)}+\|L_{n,\infty}^{\Bbb M}(f)-V_{2^{s-1}}(f)\|_{L_{\infty}(\Bbb M)}
  \\&\lesssim n^{-r+d\left(\frac1p-\frac1q\right)_+}\|f\|_{X_{p}^r(\Bbb M)}+\kappa n^{-r+d/p}\|f\|_{X_{p}^r(\Bbb M)}
  \\&\lesssim  (1+\kappa) n^{-r+d/p}\|f\|_{X_{p}^r(\Bbb M)},
\end{align*}which proves \eqref{3.2}.

This proof of Lemma \ref{thm3.3} is now finished.
\end{proof}

\section{Optimality}\label{sect4}

This section is devoted to elaborating the optimality of the
obtained results in the previous sections by giving  sharp
estimates of sampling numbers and optimal quadrature errors.

\subsection{Optimal recovery and quadrature}\label{sub4.1}
\

First we  recall some basic concepts concerning with optimal
recovery and quadrature.

Let $1\le p,q\le\infty$, $0<\tau\le\infty$, and $r>d/p$. We denote
by $BX_p^r(\Bbb M)$ the unit ball of $X_p^r(\Bbb M)$. Consider the
problem of recovering a function $f$ from $BX_p^r(\Bbb M)$
measured in $L_q(\Bbb M)$ using only finitely many function
values. The minimal worst case error that can be achieved by a
given  finite set of points $P\subset \Bbb M$ with the cardinality
$\#P$  is defined by
$$
 e(P;\,BX_p^r(\Bbb M), L_q(\Bbb M)):=\inf_{A_P}\sup_{f\in BX_p^r(\Bbb M)}\|f-A_P(f)\|_{L_q(\Bbb M)},
$$where the infimum is taking over all sampling operators of the form
 $$A_P:BX_p^r(\Bbb M)\to L_q(\Bbb M),\ \ \ A_P(f):=\varphi\left(f|_P\right),$$ and $\varphi$ is an arbitrary mapping from $\Bbb R^{\#P}$ to
$L_q(\Bbb M)$. Note that
$f|_P=\left(f(\boldsymbol\xi)\right)_{\boldsymbol\xi\in P}$ is the
restriction of $f$ to the point set $P$. Sometimes it may be
preferable to all only linear mappings $$A_P^{\rm
lin}(f)=\varphi^{\rm lin}(f|_P)=\sum_{\boldsymbol\xi\in
P}f(\boldsymbol\xi)\,a_{\boldsymbol\xi},\ \ a_{\boldsymbol\xi}\in
L_q(\Bbb M),
$$in which we write $e^{\rm lin}$ instead of $e$. Furthermore, if
we impose a constraint on the cardinality of sampling points $P$,
allowing at most $N$ points, we can define the \emph{$N$-th
sampling number (or the optimal recovery)} of $BX_p^r(\Bbb M)$
measured in the $L_q(\Bbb M)$ norm by
$$g_N(BX_p^r(\Bbb M),L_q(\Bbb M)) \ :=\inf_{\#P\le N}
e(P;\,BX_p^r(\Bbb M),L_q(\Bbb M)).
$$The \emph{$N$-th linear  sampling number} $g_N^{\rm lin}(BX_p^r(\Bbb M),L_q(\Bbb M))$ can be defined similarly.

Consider  the problem of approximating integration  of a function
$f$ from $BX_p^r(\Bbb M)$. The minimal worst case error of
$BX_p^r(\Bbb M)$ that can be achieved by a given  finite set of
points $P\subset\Bbb M$ is denoted by
$$e(P;\,BX_p^r(\Bbb M),{\rm INT}):=\inf_{A_P}\sup_{f\in BX_p^r(\Bbb M)}\left|\int_{\Bbb M}f({\bf x}){\rm d}\mu({\bf x})-A_P(f)\right|,
$$where  the infimum ranges over all sampling operators of the form
 $$A_P:BX_p^r(\Bbb M)\to \Bbb R,\ \ \ A_P(f):=\varphi\left(f|_P\right).$$
  The \emph{$N$-th optimal quadrature error} of  $BX_p^r(\Bbb M)$ are defined by
$$e_N(BX_p^r(\Bbb M),{\rm INT}):=\inf_{\#P\le N}
e(P;\,BX_p^r(\Bbb M),{\rm INT}),
$$ where the infimum is taken over all sampling points $P\subset\Bbb M$ with at most $N$ points. Note that  $BX_p^r(\Bbb M)$ is convex and balanced,
the infimum will not change when we only minimize over all linear
mappings $\varphi^{\rm lin}$. Therefore, it is no restriction to
assume that $A_P$ is of the form
$$A_P^{\rm lin}(f)=\sum_{\boldsymbol\xi\in P}\lambda_{\boldsymbol\xi}f({\boldsymbol\xi}),\ \lambda_{\boldsymbol\xi}\in\Bbb R.$$
 This is a classical result by Smolyak and Bakhvalov (see, e.g. \cite[Theorem 4.7]{NW}).

It is well known (see e.g \cite{TWW}) that
\begin{equation}\label{4.1}
  g_N^{\rm lin}(BX_p^r(\Bbb M),L_q(\Bbb M))\ge g_N(BX_p^r(\Bbb M),L_q(\Bbb
  M)),
\end{equation}and
\begin{align}\label{4.2}
 g_N(BX_p^r(\Bbb M),L_q(\Bbb M))
\ge\inf_{\substack{\boldsymbol\xi_1,\dots,\boldsymbol\xi_N\in \Bbb M}}\sup_{\substack{f\in BX_p^r(\Bbb M)\\
  f(\boldsymbol\xi_1)=\dots=f(\boldsymbol\xi_N)=0}}\|f\|_{L_q(\Bbb
  M)}.
\end{align}
   Meanwhile, it is evident (see e.g \cite{TWW}) that
 \begin{equation}\label{4.3}
   e_N(BX_p^r(\Bbb M),{\rm INT})\ge \inf_{\substack{\boldsymbol\xi_1,\dots,\boldsymbol\xi_N\in \Bbb M}}
   \sup_{\substack{f\in BX_p^r(\Bbb M)\\ f(\boldsymbol\xi_1)=\dots=f(\boldsymbol\xi_N)=0}}\left|\int_{\Bbb M}f({\bf x}){\rm d}\mu({\bf
   x})\right|.
 \end{equation}
\vskip2mm The study of optimal recovery and  quadratures for
Sobolev and Besov spaces is of significant importance and
relevance in both theoretical and practical contexts. In recent
years, extensive research has been conducted on these topics, for
example,  for the case of the torus $\mathbb{T}^d$, see
\cite{D,FHL,T1,T2,T3};  for the case of the sphere $\mathbb{S}^d$,
see \cite{BDSSWW,BH,BSSW,DW,He2,HMS,HS,HS2,HS3,HW,KN,W2,WS,WW}; for
other domains, see \cite{BCCGT,BCCGST,KPUU1,KPUU2,KS}.

An outstanding result in this field for the case of the compact
Riemannian manifold $\Bbb M$ is the work by D. Krieg and M.
Sonnleitner \cite{KS}, which characterized  optimal recovery and
quadrature for the Sobolev spaces using general sampling point set
$P$ by the distance function from $P$. More specifically, their
results can be formulated as follows: Let $1<p<\infty$, $1\le
q\le\infty$, and $r>d/p$. Then for any nonempty and finite set $P$
of points in $\Bbb M$, we have
\begin{equation}\label{4.4}
e(P;\,BH_p^r(\Bbb M), L_q(\Bbb M))\asymp e^{\rm lin}(P;\,BH_p^r(\Bbb M), L_q(\Bbb M))
\asymp \left\|{\rm dist}_{\Bbb M}(\cdot,P)\right\|_{L_\gamma(\Bbb M)}^{r-d\left(\frac1p-\frac1q\right)_+},
\end{equation}and
\begin{equation}\label{4.5}
  e(P;\,BH_p^r(\Bbb M), {\rm INT})\asymp  \left\|{\rm dist}_{\Bbb M}(\cdot,P)\right\|_{L_{rp^*}(\Bbb M)}^{r},
\end{equation}
where  the implied constants are independent of $P$, $1/p+1/p^*=1$, and
$$ \gamma:=\Bigg\{\begin{aligned}&r(1/q-1/p)^{-1},  &&q<p,
  \\&\infty, &&q\ge p.\end{aligned}$$
Meanwhile, it follows from the Ahlfors-regularity that for all $0<\gamma\le\infty$ and $\alpha>0$,
$$\inf_{\#P\le N}\left\|{\rm dist}_{\Bbb M}(\cdot,P)\right\|_{L_\gamma(\Bbb M)}^\alpha\asymp N^{-\alpha/d},$$
where the infimum is taken over all point sets of cardinality at
most $N$. Together with \eqref{4.4} and \eqref{4.5}, it deduces
that for $1<p<\infty$, $1\le q\le\infty$, and $r>d/p$,
\begin{equation}\label{4.6}
g_N(BH_p^r(\Bbb M),L_q(\Bbb M))\asymp g_N^{\rm lin}(BH_p^r(\Bbb M),L_q(\Bbb M))\asymp N^{-\frac rd+\left(\frac1p-\frac1q\right)_+},
\end{equation}
and
\begin{equation}\label{4.7}
e_N(BH_p^r(\Bbb M),{\rm INT})\asymp  N^{-r/d}.
\end{equation}
Indeed, \eqref{4.7} still holds for $p\in\{1,\infty\}$ (see \cite{BCCGST}). However,
the validity of \eqref{4.6} for $p\in\{1,\infty\}$ remains  open.

In this section, we will give  sharp estimates of sampling numbers
for the Sobolev classes $BH_p^r(\Bbb M)$. We also obtain relevant
results for the Besov classes $BB_{p,\tau}^r(\Bbb M)$.

\begin{thm}\label{thm4.1} Let $1\le p,q\le \infty$, $0<\tau\le\infty$. Then

 (i) for $r>d/p$, we have $$g_N(BX_p^r(\Bbb M),L_q(\Bbb M))\asymp N^{-\frac rd+\left(\frac1p-\frac1q\right)_+};$$

(ii) if $2$ is between $p$ and $q$, and if $r>d\max\{1/p,1/2\}$,
then we have
$$g_N^{\rm lin}(BX_p^r(\Bbb M),L_q(\Bbb M))\asymp N^{-\frac
rd+\left(\frac1p-\frac12\right)_+}.$$
\end{thm}

\begin{thm}\label{thm4.3} Let $1\le p\le \infty$, $0<\tau\le\infty$, and $r>d/p$. Then
$$e_N(BX_p^r(\Bbb M),{\rm INT})\asymp  N^{-r/d}.$$
\end{thm}

\subsection{Proofs of Theorems \ref{thm4.1} and \ref{thm4.3}}
\

This subsection is dedicated to proving  the sharp estimates of
sampling numbers and optimal quadratures, which have been
partially proven in previous studies. However, in order to ensure
a comprehensive understanding, we will provide a complete proof
for both the upper and lower estimates in different and straight
way. The proof of the upper estimates is based on Lemma 3.1.
 The proof of the lower estimates relies on the well-known technique of fooling functions. Namely,
 we provide a function belonging to $ BX_p^r(\Bbb M)$ that vanishes on the given points set $X_N$ as constructed
 below.

\vskip2mm

\noindent\textbf{Construction of bump functions:}

Let $\Xi$ be a maximal $\vz/n$-separated subset with the
cardinality $2C_dN\le \#\Xi\asymp N$ in $\Bbb M$, where $C_d>0$ is
the constant given in \eqref{2.4-0}, and $\vz$ is a sufficiently
small positive number. Then $N\asymp n^d$. For an arbitrary finite
set of points $X_N:=\{\boldsymbol\xi_j:\,j=1,\dots,N\}$ in $\Bbb
M$, we set
$$G:=\left\{\boldsymbol\eta_j:\,\boldsymbol\xi_i\notin {\rm B}\left(\boldsymbol\eta_j,\vz/n\right),\ {\rm for\ all}\ i=1,\dots,N\right\}.$$
By \eqref{2.4-0} we have
$$\#G\ge 2C_dN-C_dN\ge C_dN\ge N.$$ Without loss of generality, we assume that $\boldsymbol\eta_1,\boldsymbol\eta_2,\dots,\boldsymbol\eta_N\in G$.
Then for each $j=1,\dots, N$, we can construct a bump function
$\psi_j({\bf x})$ supported on ${\rm
B}\left(\boldsymbol\eta_j,\vz/n\right)$ with
 $$\left\|(-\Delta_{\Bbb M})^k\psi_j\right\|_{L_p(\Bbb M)}\le CN^{2k/d},\ \ k=0,1,\dots,$$
 and
 $$\int_{\Bbb M}\psi_j({\bf x}){\rm d}\mu({\bf x})=N^{-1}.$$
Thus, the bump functions $\psi_j,j=1,\dots,N$ have disjoint
supports and vanish on $X_N$.

The construction of such functions in $\Bbb R^d$ can be done by
translating and dilating any fixed nonnegative $C^\infty$ smooth
function with compact support and integral $1$. Then one can
transport and normalize these functions to the local charts of the
manifold and, by compactness, the constant of $C$ can be chosen
independent of $j$ and $N$. See e.g. \cite[Lemma 2]{W2} for the
case of $\Bbb S^d$.

Also, it can be prove that for all $1\le p,q\le\infty$ and $r>0$,
\begin{equation}\label{1}
\left\|\psi_j\right\|_{H_p^r(\Bbb M)}\lesssim N^{r/d-1/p},\ \
\|\psi_j\|_{L_q(\Bbb M)}\asymp N^{-1/q}.
\end{equation}
See e.g. \cite{BCCGT} for details.

Moreover, we claim that for all $1\le p\le\infty$,
$0<\tau\le\infty$, and $r>0$,
\begin{equation}\label{2}
\left\|\psi_j\right\|_{B_{p,\tau}^r(\Bbb M)}\lesssim N^{r/d-1/p}.
\end{equation}
To prove this claim, we note that $E_{2^j}(\psi_j)_{p}\le
\|\psi_j\|_{L_p(\Bbb M)}$ for any $j\ge0$. Then by \eqref{1} we
obtain
\begin{align}\label{3}
  \sum\limits_{2^j<N^{1/d}}\Big(2^{jr} E_{2^j}(\psi_j)_{p}\Big)^\tau&\le \|\psi_j\|_{L_p(\Bbb M)}^\tau\sum\limits_{2^j<N^{1/d}}2^{jr\tau}\notag
  \\&\lesssim N^{r\tau/d-\tau/p}.
\end{align}
By \eqref{2.12}, we can choose a positive number $\upsilon>r$ such
that for any $j\ge0$,
 $$E_{2^j}(\psi_j)_{p}\lesssim 2^{-j\upsilon}\|\psi_j\|_{H_{p}^\upsilon(\Bbb M)},$$which, by \eqref{1}, leads to
 \begin{align}\label{4}
  \sum\limits_{2^j\ge N^{1/d}}\Big(2^{jr} E_{2^j}(\psi_j)_{p}\Big)^\tau&
  \le \|\psi_j\|_{H_{p}^\upsilon(\Bbb M)}^\tau\sum\limits_{2^j\ge N^{1/d}}2^{j(r-\upsilon)\tau}\nonumber
  \\&\asymp N^{(r-\upsilon)\tau/d}\|\psi_j\|_{H_{p}^\upsilon(\Bbb M)}^\tau\notag
  \\&\lesssim N^{r\tau/d-\tau/p}.
\end{align}
Combing  \eqref{3} with \eqref{4}, we get \eqref{2}.

Define $\psi({\bf x})=\sum_{j=1}^N\psi_j({\bf x})$. Then
$\psi|_{X_N}=0$. Furthermore, since $(I-\Delta_{\Bbb
M})^{r/2}\psi_j$, $j=1,\dots,N$ have disjoint supports, we can
immediately obtain
\begin{equation}\label{5}
\|\psi\|_{H_{p}^r(\Bbb M)}\lesssim N^{r/d},\ \ \ \int_{\Bbb
M}\psi({\bf x}){\rm d}\mu({\bf x})=1,\ \ \ \|\psi\|_{L_p(\Bbb
M)}\asymp 1.
\end{equation}
Meanwhile, by the same methods, we can prove that
\begin{equation}\label{6}
\|\psi\|_{B_{p,\tau}^r(\Bbb M)}\lesssim N^{r/d}.
\end{equation}We keep the above notations in the following.

\

\noindent{\it Proof of Theorem \ref{thm4.1}.}\

\emph{\textbf{Lower estimates:}} It follows from \eqref{4.2} that
$$ g_N(BX_p^r(\Bbb M),L_q(\Bbb M))\ge\inf_{\substack{\boldsymbol\xi_1,\dots,\boldsymbol\xi_N\in \Bbb M}}\sup_{\substack{f\in BX_p^r(\Bbb M)\\
  f(\boldsymbol\xi_1)=\dots=f(\boldsymbol\xi_N)=0}}\|f\|_{L_q(\Bbb
  M)}.
$$ Therefore, for  arbitrary $N$ points $\boldsymbol\xi_j,\,j=1,\dots,N$ in $\Bbb
M$, it suffices to show that there exist functions $f_0$ with
$$f_0(\boldsymbol\xi_j)=0,\ \ j=1,\dots,N,$$
such that
$$
\|f_0\|_{X_p^r(\Bbb M)}\lesssim 1\ \  \ {\rm and}\ \ \
\|f_0\|_{L_q(\Bbb M)}\asymp N^{-\frac
rd+\left(\frac1p-\frac1q\right)_+}.
$$
The following proof will be divided into two cases.

\emph{Case 1: $1\le q\le p\le\infty$.}

Let $f_0({\bf x})=N^{-r/d}\psi({\bf x})$. Then
$f_0(\boldsymbol\xi_j)=0,\ \ j=1,\dots,N$. Meanwhile, by \eqref{5}
and \eqref{6} we have
$$\|f_0\|_{X_p^r(\Bbb M)}\lesssim 1\ \  \ {\rm and}\ \ \ \|f_0\|_{L_q(\Bbb M)}\asymp N^{-\frac rd},$$
which  proves the lower estimates for  $1\le q\le p\le\infty$.

\emph{Case 2: $1\le p<q\le\infty$.}

Let $f_0({\bf x})=N^{-r/d+1/p}\psi_1({\bf x})$. Then
$f_0(\boldsymbol\xi_j)=0,\ \ j=1,\dots,N$. Meanwhile, by \eqref{1}
and \eqref{2} we have
$$\|f_0\|_{X_p^r(\Bbb M)}\lesssim 1\ \  \ {\rm and}\ \ \ \|f_0\|_{L_q(\Bbb M)}\asymp N^{-\frac rd+\left(\frac1p-\frac1q\right)},$$
which proves the lower estimates for  $1\le p<q\le\infty$.

\emph{\textbf{Upper estimates:}} According to Lemma \ref{lem4.5},
we know that for any $1\leq p\leq\infty$, there exists an
$L_p$-Marcinkiewicz-Zygmund family $\mathcal{X}$ on $\mathbb{M}$
with $N_n\asymp N\asymp n^{d}$. Let $L_{n,p}$ be the weighted
least $\ell_p$ approximation operator induced by $\mathcal{X}$.
Then for $1\leq p,q\le\infty$, $r>d/{p}$,
\begin{align*}
  g_N(BX_{p}^{r}(\mathbb{M}),L_{q}(\mathbb{M}))&\le\sup_{f\in BX_{p}^{r}(\mathbb{M})}\|f-L_{n,p}^{\mathbb{M}}(f)\|_{L_q(\mathbb{M})}
  \\&\lesssim n^{-r+d\left(\frac{1}{p}-\frac{1}{q}\right)_+}\asymp N^{-\frac rd+\left(\frac{1}{p}-\frac{1}{q}\right)_+}.
\end{align*}
In particular, assume that  2 is between $p$ and $q$, and
$r>d\max\{1/p,1/2\}$. Since $L_{n,2}^{\Bbb M}$ is linear,  Lemma \ref{thm3.3} gives
\begin{align*}
  g_N^{\rm lin}(BX_{p}^{r}(\mathbb{M}),L_q(\mathbb{M}))&\le\sup_{f\in BX_{p}^{r}(\mathbb{M})}\|f-L_{n,2}^{\mathbb{M}}(f)\|_{L_q(\mathbb{M})}
  \\&\lesssim n^{-r+d\left(\frac{1}{p}-\frac{1}{q}\right)_+}\asymp N^{-\frac rd+\left(\frac{1}{p}-\frac{1}{q}\right)_+}.
\end{align*}Together with \eqref{4.1}, we obtain the upper estimate of Theorems \ref{thm4.1}.

The proof of Theorem \ref{thm4.1} is completed. $\hfill\Box$

\

\noindent{\it Proof of Theorems \ref{thm4.3}.} \

 It suffices to verify
that for $1\le p\le \infty$, $0<\tau\le\infty$, and $r>d/p$,
$$e_N(BB_{p,\tau}^r(\Bbb M),{\rm INT})\asymp  N^{-r/d}.$$

For this purpose, let $X_N:=\{\boldsymbol\xi_j:\,j=1,\dots,N\}$ be
any subset of  $\Bbb M$. Based on the construction of bumping
functions in the front, we set $f_0({\bf x})=N^{-r/d}\psi({\bf
x})$. Then  $f_0|_{X_N}=0$,
$$\|f_0\|_{B_{p,\tau}^r(\Bbb M)}\lesssim 1\ \ {\rm and}\ \  \int_{\Bbb M}f_0({\bf x}){\rm d}\mu({\bf x})\asymp N^{-r/d}.$$
Therefore, we can derive from \eqref{4.3} that
\begin{align}
e_N(BB_{p,\tau}^r(\Bbb M),{\rm INT})&\ge \inf_{X_N\subset \Bbb M}\sup_{\substack{f\in BB_{p,\tau}^r(\Bbb M)\\
  f|_{X_N}=0}}
\left|\int_{\Bbb M}f({\bf x}){\rm d}\mu({\bf x})\right|\notag\\
&\gtrsim \inf_{X_N\subset \Bbb M}\left|\int_{\Bbb M}f_0({\bf
x}){\rm d}\mu({\bf x})\right|\gtrsim  N^{-r/d}. \label{4.21}
\end{align}

On the other hand, according to Lemma \ref{lem4.5}, we know that
for any $1\leq p\leq\infty$, there exists an
$L_p$-Marcinkiewicz-Zygmund family $\mathcal{X}$ on $\mathbb{M}$
with $N_n\asymp N\asymp n^{d}$. Let $L_{n,p}$ be the weighted
least $\ell_p$ approximation operator induced by $\mathcal{X}$ and
let
$$A_{X_N}(f):=\int_{\Bbb M}L_{n,p}^{\Bbb M}(f)({\bf x}){\rm d}\mu({\bf x}).$$Then by the H\"older inequality and Theorem \ref{rem1.1} we have
\begin{align*}
  e_N(BB_{p,\tau}^r(\Bbb M),{\rm INT})&\le\sup_{f\in BB_{p,\tau}^r(\Bbb M)}\left|\int_{\Bbb M}f({\bf x}){\rm d}\mu({\bf x})-A_{X_N}(f)\right|
  \\&\le\sup_{f\in BB_{p,\tau}^r(\Bbb M)}\|f-L_{n,p}^{\Bbb M}(f)\|_{L_p(\Bbb M)}
  \\&\lesssim n^{-r}\asymp N^{-r/d},
\end{align*}
which, combining with \eqref{4.21}, completes the proof.
$\hfill\Box$

\subsection{Optimality}\

Let us now proceed to discuss the optimality  of the weighted
least $\ell_p$ approximation and the least squares quadrature.
According to Lemma \ref{lem4.5}, we know that for any $1\leq
q\leq\infty$, there exists an $L_q$-Marcinkiewicz-Zygmund family
on $\mathbb{M}$ with $N_n\asymp N\asymp n^{d}$. Based on Theorems
\ref{thm1.1} and \ref{rem1.1}, we can conclude that for $1\leq
p,q\leq \infty$, $r>d/{p}$,
$$\sup_{f\in
BX_{p}^{r}(\mathbb{M})}\|f-L_{n,p}^{\mathbb{M}}(f)\|_{L_q(\mathbb{M})}
\asymp
N^{-\frac{r}{d}+\left(\frac{1}{p}-\frac{1}{q}\right)_+}\asymp
g_N(BX_{p}^{r}(\mathbb{M}),L_{q}(\mathbb{M})),$$and for $1\leq
p,q< \infty$, $r>d\max\{1/p,1/q\}$,
$$\sup_{f\in BX_{p}^{r}(\mathbb{M})}\|f-L_{n,q}^{\mathbb{M}}(f)\|_{L_q(\mathbb{M})}
\asymp
N^{-\frac{r}{d}+\left(\frac{1}{p}-\frac{1}{q}\right)_+}\asymp
g_N(BX_{p}^{r}(\mathbb{M}),L_{q}(\mathbb{M})).$$ This implies that
the weighted least $\ell_p$ (or $\ell_q$) approximation operator
$L_{n,p}^{\mathbb{M}}$ (or $L_{n,q}^{\mathbb{M}}$) serves as an
asymptotically optimal algorithm for optimal recovery of
$BX_{p}^{r}(\mathbb{M})$ measured in the $L_q(\Bbb M)$ norm,
subject to certain restrictions on $r$.

Also, assume that  2 is between $p$ and $q$, and
$r>d\max\{1/p,1/2\}$. It follows from Lemma 3.1  that
$$\sup_{f\in
BX_{p}^{r}(\mathbb{M})}\|f-L_{n,2}^{\mathbb{M}}(f)\|_{L_q(\mathbb{M})}
\asymp
N^{-\frac{r}{d}+\left(\frac{1}{p}-\frac{1}{q}\right)_+}\asymp
g_N^{\rm lin}(BX_{p}^{r}(\mathbb{M}),L_{q}(\mathbb{M})).$$
 This implies that the weighted least squares
approximation operator $L_{n,2}^{\mathbb{M}}$ is an asymptotically
optimal linear algorithm for optimal recovery of
$BX_{p}^{r}(\mathbb{M})$ measured in  the $L_q(\Bbb M)$ norm on
the condition that 2 is between $p$ and $q$, and
$r>d\max\{1/p,1/2\}$.

 For the least squares quadrature rule $I_n^{\mathbb{M}}$, it can be deduced from Theorem \ref{thm4.3}
 that for $2\leq p\leq\infty$ and $r>d/{2}$,
 $$\sup_{f\in BX_p^{r}(\mathbb{M})}\left|\int_{\mathbb{M}}f({\bf x}){\rm d}\mu({\bf x})-I_n^{\mathbb{M}}(f)
 \right|\asymp N^{-r/{d}}\asymp e_N(BX_p^{r}(\mathbb{M}),\rm{INT}).$$
  This indicates that the least squares quadrature rule $I_n^{\mathbb{M}}$ is asymptotically optimal quadrature
  rule for $BX_p^{r}(\mathbb{M})$ on
the condition that $2\leq p\leq\infty$ and $r>d/{2}$.

\vskip2mm We finish this section by considering two special cases
of compact Riemannian manifolds.
\begin{exam}Let $1\le p,q\le\infty$, $0<\tau\le \infty$, $r>d/p$, and let $\Delta:=\sum_{j=1}^d\partial^2/\partial x_j^2$ be the Laplace operator
on the $d$-torus $\Bbb T^d$ with eigenvalues $\{-4\pi^2\|{\bf
k}\|^2\}_{{\bf k}\in\Bbb Z^d}$ and eigenfunctions $\{\exp(2\pi
{\bf k}{\bf x})\}_{{\bf k}\in\Bbb Z^d}$. Let $L_{n,p}^{\Bbb T^d}$
be the weighted least $\ell_p$ approximation operator induced by
an $L_{p}$-MZ family $\mathcal X$ on $\Bbb T^d$, and let
$I_n^{\Bbb T^d}$ be the least squares quadrature rule induced by
an $L_{2}$-MZ family $\tilde {\mathcal X}$ on $\Bbb T^d$. Then
applying Theorems \ref{thm1.1} and \ref{thm1.2}, we have
$$\|f-L_{n,p}^{\Bbb T^d}(f)\|_{L_q(\Bbb T^d)}=\mathcal{O}\left(n^{-r+d\left(\frac1p-\frac1q\right)_+}\right),$$
for  approximating  $f\in BX_p^r(\Bbb T^d)$ from the samples
$\{{\bf x}_{n,k}:\,k=1,\dots,N_n\}$ of the $n$-th layer of
$\mathcal X$, and
$$\left|\int_{\Bbb T^d}f({\bf x}){\rm d}{\bf x}-I_n^{\Bbb
T^d}(f)\right|=\mathcal{O}\left(n^{-r}\right),$$ for approximating
$\int_{\Bbb T^d}f({\bf x}){\rm d}{\bf x}$,  $f\in BX_p^r(\Bbb
T^d)$, $2\le p\le \infty$ from the samples $\{\tilde{\bf
x}_{n,k}:\,k=1,\dots,\tilde N_n\}$ of the $n$-th layer of $\tilde
{\mathcal X}$. These estimates are optimal, which improves
\cite[Equations (4) and (22)]{G} as a special case $d=1$.
\end{exam}

\begin{exam} Let $1\le p,q\le\infty$, $0<\tau\le \infty$, $r>d/p$,  and let $\Delta_0$ be the Laplace-Beltrami operator
 on the sphere $\Bbb S^d$ with eigenvalues $\{-n(n+d-1)\}_{n=0}^\infty$ and
 eigenfunctions being
  the restriction to the sphere of homogenous harmonic polynomials in $\Bbb R^{d+1}$ (see
  \cite{DaX}). Let $L_{n,p}^{\Bbb S^d}$
be the weighted least $\ell_p$ approximation operator induced by
an $L_{p}$-MZ family $\mathcal X$ on $\Bbb S^d$, and let
$I_n^{\Bbb S^d}$ be the least squares quadrature rule induced by
an $L_{2}$-MZ family $\tilde {\mathcal X}$ on $\Bbb S^d$.
  Then applying Theorems \ref{thm1.1} and \ref{thm1.2}, we have
$$\|f-L_{n,p}^{\Bbb S^d}(f)\|_{L_q(\Bbb S^d)}=\mathcal{O}\left(n^{-r+d\left(\frac1p-\frac1q\right)_+}\right),$$
for  approximating $f\in BX_p^r(\Bbb S^d)$ from the samples
$\{{\bf x}_{n,k}:\,k=1,\dots,N_n\}$ of the $n$-th layer of
$\mathcal X$, and
 $$\left|\int_{\Bbb S^d}f({\bf x}){\rm d}\sigma({\bf x})-I_n^{\Bbb S^d}(f)\right|=\mathcal{O}\left(n^{-r}\right),$$
for approximating $\int_{\Bbb S^d}f({\bf x}){\rm d}\sigma({\bf
x})$,  $f\in BX_p^r(\Bbb S^d)$, $2\le p\le \infty$ from the
samples $\{\tilde{\bf x}_{n,k}:\,k=1,\dots,\tilde N_n\}$ of the
$n$-th layer of $\tilde {\mathcal X}$. These estimates are
optimal, which generalizes the results in \cite{LW}.
\end{exam}

\section{Concluding remarks}\label{sect5}

In this paper, we consider the weighted least $\ell_p$
approximation problem on compact Riemannian manifolds and derive
the approximation theorem  for the weighted least $\ell_p$
approximation for both Sobolev and Besov classes, which are
asymptotically optimal. The proofs are straightforward, although
it is worth noting their limitations.

Firstly, the proofs are specifically designed for the Riemannian
measure on compact Riemannian manifolds. Different techniques may
be required for general probability measures.  Secondly, we aspire
to extend our results to general metric spaces where the existence
of $L_p$-Marcinkiewicz-Zygmund families has been established by
Filbir and Mhaskar in a series of papers \cite{FM1,FM,MM,Mh2}.

\

\noindent\textbf{Acknowledgments}
  The authors  were
supported by the National Natural Science Foundation of China
(Project no. 12371098).

%%===========================================================================================%%
%% If you are submitting to one of the Nature Portfolio journals, using the eJP submission   %%
%% system, please include the references within the manuscript file itself. You may do this  %%
%% by copying the reference list from your .bbl file, paste it into the main manuscript .tex %%
%% file, and delete the associated \verb+\bibliography+ commands.                            %%
%%===========================================================================================%%

\bibliography{sn-bibliography}% common bib file
%% if required, the content of .bbl file can be included here once bbl is generated
%%\input sn-article.bbl

\end{document}